\documentclass[10pt]{article} 
\usepackage{a4,amssymb, amsmath, amsthm}
\usepackage{graphics,color}
\usepackage{epsfig}
\usepackage{subfigure}
\usepackage{caption}
\usepackage{ulem}
\usepackage{multirow}
\usepackage{verbatim}

\newtheorem{theorem}{Theorem}[section]
\newtheorem{corollary}[theorem]{Corollary}
\newtheorem{lemma}[theorem]{Lemma}
\newtheorem{proposition}[theorem]{Proposition}

\theoremstyle{definition}
\newtheorem{definition}[theorem]{Definition}

\newcommand*{\bbar}[1]{\bar{\bar{#1}}}

\begin{document}
\normalem
\providecommand{\keywords}[1]{{\noindent \textbf{Keywords:}} #1}
\providecommand{\msc}[1]{{\noindent \textit{Mathematics Subject Classification:}} #1}

\title{A space-time adaptive boundary element method \\for the wave equation}
\author{*A. Aimi, *G. Di Credico, $^\S$H. Gimperlein, *C. Guardasoni\\\\\small *Dept.~of Mathematical Physical and Computer Sciences, University of Parma, Italy\\\small *Members of the INdAM-GNCS Research Group, Italy\\\small$^\S$Engineering Mathematics, University of Innsbruck, Austria}
\date{}
\maketitle \vskip 0.5cm
\begin{abstract}
\noindent This article initiates the study of space-time adaptive mesh refinements for time-dependent boundary element formulations of wave equations. Based on error indicators of residual type, we formulate an adaptive boundary element procedure for acoustic soft-scattering problems with local tensor-product refinements of the space-time mesh. We discuss the algorithmic challenges and investigate the proposed method in numerical experiments. In particular, we study the performance and improved convergence rates with respect to the energy norm for problems dominated by spatial, temporal or traveling singularities of the solution. The efficiency of the considered rigorous and heuristic a posteriori error indicators is discussed.
\end{abstract}

\section{Introduction}

For time-independent problems with singular solutions, adaptive mesh refinements give rise to efficient versions of both finite element and boundary element methods \cite{bonito,gwinsteph}, with improved or optimal convergence rates.  Correspondingly, for time-dependent problems space- or time-adaptive boundary element methods have attracted much recent interest \cite{gantner,review,adaptive,hoonhout2023,zank1d}. However, meshes which are locally refined in both space and time are crucial to resolve space-time singularities such as traveling wave fronts or singularities in nonlinear problems \cite{contact}. Partly due to the algorithmic and analytic challenges, such fully space-time adaptive methods have hardly been explored for hyperbolic problems \cite{Glaefke}. 
 
In this article we initiate the study of fully space-time adaptive mesh refinement procedure for the wave equation, formulated as a boundary integral equation in the time-domain \cite{costabel04, sa}. Based on a posteriori error estimates of residual type \cite{adaptive}, the proposed adaptive mesh refinements follows the four steps:
 \begin{align*}
  \textbf{SOLVE}&\longrightarrow \textbf{ESTIMATE}\longrightarrow \textbf{MARK}\longrightarrow \textbf{REFINE}.
\end{align*}


We here present this fully space-time  adaptive method, discuss the involved algorithmic challenges and investigate its properties in numerical experiments.

To describe the main results, we consider the following model problem for the acoustic wave equation 


\begin{equation}\label{eq_wave}
\partial_t^2 u - \Delta u = 0\ , \quad u=0\ \text{ for }\ t\leq 0\ ,    
\end{equation}


\noindent in the exterior outside an (open or closed) curve $\Gamma\subset\mathbb{R}^2$. The soft-scattering problem imposes inhomogeneous Dirichlet boundary conditions 


\begin{equation}\label{Dir_cond}
u  = f 
\end{equation}


\noindent on the obstacle $\Gamma$.\\
Following \cite{bh}, the problem \eqref{eq_wave}, \eqref{Dir_cond} is equivalent to a time dependent weakly singular integral equation for an unknown density $\psi$ on $\Gamma$
\begin{align}\label{weaklysingeq}
\mathcal{V} \psi(t,x) = \iint_{\mathbb{R}^+\times \Gamma} G(t- \tau,x,y)\ \psi(\tau,y)\ d\gamma_y \ d\tau= f(t,x)\,,
\end{align}
which involves the fundamental solution 
\begin{align}\label{eq:green}
 G(t-\tau,x,y)&= \frac{H(t-\tau-|x-y|)}{2\pi \sqrt{(t-\tau)^2-|x-y|^2}}
\end{align}


\noindent of the wave equation in $\mathbb{R}^2$.

Based on a weak formulation of \eqref{weaklysingeq} related to the energy \cite{ADGMP2009}, we consider Galerkin discretizations using tensor products of piecewise polynomials in space and time in each space-time element.

Error indicators based on available residual a posteriori error estimates \cite{adaptive} are used to introduce the space-time adaptive algorithm in Subsection \ref{subsec:stadaptive}. We address the algorithmic challenges compared to previously studied adaptive methods in space or time separately, which, in particular, exploited the global tensor product structure of the mesh. The proposed adaptive algorithm is studied in numerical experiments. They illustrate the improved convergence rates with respect to the energy norm for problems dominated by spatial, temporal or traveling singularities, with reductions in the required number of degrees of freedom and memory. The experiments indicate the efficiency and reliability of the error estimates in appropriate space-time norms.

The current work contributes to the recent interest in space-time adaptive boundary element methods for the wave equation.
For the Dirichlet problem considered in this article, a posteriori error estimates were studied in \cite{adaptive}. However, they were only used for space-adaptive mesh refinements with a uniform time step, while the challenges of fully space-time adaptive refinements were described. Gl\"{a}fke \cite{Glaefke} obtained first computational results towards space-time refinements in $\mathbb{R}^2$. Unpublished work by Abboud uses error {estimators of Zienkiewicz-Zhu type \cite{zz}, as often used in computational engineering,} towards space-adaptive mesh refinements for screen problems in $\mathbb{R}^3$. {For the  Neumann problem, space-adaptive mesh refinements were recently considered in \cite{aimi2024W}.}

The literature on adaptive time discretizations to resolve singular temporal behavior is more limited. For the soft scattering problem in $\mathbb{R}$ the adaptive selection of time steps was recently studied in \cite{hoonhout2023,zank2020inf,zank1d}, following earlier work \cite{sv} in $\mathbb{R}^3$
. For time discretizations using convolution quadrature non-uniform time steps have been of much interest. {For wave equations, references \cite{ls,lf2015} provided a framework for numerically evaluating the convolution integral with the possibility to accommodate adaptive time stepping. We refer to \cite{cicci, menon} for first works on adaptive time stepping based on generalized convolution quadrature to solve boundary integral formulations of the time-dependent wave equation.}

Beyond adaptive methods, both Galerkin and convolution quadrature methods have attracted much interest for wave equations. Such time domain methods are of particular relevance for problems which cannot be reduced to the frequency domain, including nonlinear problems and problems that involve a broad range of frequencies. We refer to   \cite{banjai2022integral,costabel04,hd,review} for an overview. Singular solutions have been particularly studied in the case of time-independent geometric singularities, where quasi-optimal convergence rates have been shown for time-independent graded meshes in space or using $hp$-versions on quasi-uniform meshes \cite{aimi23,hp,graded}. Relevant to the current article are works on the efficient assembly and compression of the space-time matrices for both time-stepping and more general space-time discretizations \cite{aimi2020,desideriofalletta,hsiao1,polzschanz,polzschanz2,bertoluzza,merta}, also beyond the classical global tensor product meshes.\\

\noindent \emph{Structure of this article:} Section \ref{secproblem} introduces the weak formulation of integral equation \eqref{weaklysingeq} and its abstract Galerkin discretization. The considered space-time discretizations using local tensor products of piecewise polynomials in space and time are described in Section \ref{Discretization} 
together with the a posteriori error estimate which is used to define  {(so-called theoretical)} error indicators. Subsequently, algorithmic details are presented in Section \ref{algo}. 
Numerical results are collected in Section \ref{Numerical results}. They assess the performance and convergence properties of the proposed fully space-time adaptive procedure and the efficiency of the {theoretical} error indicators, {which are further compared with an alternative, heuristic error indicator.}\\

{\noindent\emph{Notation:} We write $f \lesssim g$ provided there exists a constant $C>0$ such that $f \leq Cg$. If the constant $C$ is allowed to depend on a parameter $\sigma$, we write $f \lesssim_\sigma g$.}

\section{Weak formulation of the model problem}\label{secproblem}

We assume that $\Gamma$ is the boundary of a polygonal Lipschitz domain or an open polygonal Lipschitz curve. The weak formulation of equation \eqref{weaklysingeq} involves the bilinear form 
\begin{equation}
    B(\psi, \phi) = \iint_{\mathbb{R}^+ \times \Gamma} \partial_t\mathcal{V}  {\psi}(t,x)\ \phi(t,x)\ d\gamma_x \, d_\sigma t, 
\end{equation}
where $d_\sigma t = \mathrm{e}^{-2\sigma t} dt$ for fixed $\sigma>0$. \\

\noindent {\bf Remark.} The numerical simulations are typically related to a bounded time interval of analysis, where we set $\sigma=0$ as usual \cite{ADGMP2009,bh}. \\

{For the analysis, space-time anisotropic Sobolev spaces provide a convenient framework of function spaces and related norms, going back to \cite{bh,hd}. Closely related function spaces were recently also used for the a posteriori analysis of finite element discretizations, see \cite{cf23, chaumont2024damped}. In particular, the Sobolev space $H^r_\sigma(\mathbb{R}^+,\widetilde{H}^s({\Gamma}))$ (essentially) consists of those distributions supported in $\mathbb{R}^+\times \Gamma$ such that $s$ spatial derivatives and $s+r$ time derivatives belong to $L^2(\mathbb{R}^+\times \Gamma, d_\sigma t\, d\gamma_x)$. The corresponding Sobolev norm is denoted by $\|\cdot\|_{r,s,\Gamma,\ast} $.  We refer to \cite{aimi23}, Appendix A, for precise definitions.}\\

The properties of $B(\cdot,\cdot)$ follow from the properties collected in
\begin{theorem}\label{mapthm} Let $r \in \mathbb{R}$. 
\begin{description} 
  \item[a)] Then the weakly singular operator is continuous,
\begin{equation} \label{mapping}
\mathcal{V} : H_\sigma^{r+1}(\mathbb{R}^+, \widetilde{H}^{-\frac{1}{2}}(\Gamma)) \to H_\sigma^{r}(\mathbb{R}^+, H^{\frac{1}{2}}(\Gamma))\ . \end{equation}
  \item[b)] The operator $ \partial_t \mathcal{V}$ is weakly coercive: \begin{equation}
    \iint_{\mathbb{R}^+ \times \Gamma} \partial_t(\mathcal{V} \psi(t,x)) \psi(t,x)\  d\gamma_x \ d_\sigma t \gtrsim_\sigma \|\psi\|_{0,-{\frac{1}{2}},\Gamma,\ast}^2,
\end{equation}
\noindent and the inverse $\mathcal{V}^{-1}$ is continuous, \begin{equation} \label{mapping_inverse}
\mathcal{V}^{-1} : H_\sigma^{r+1}(\mathbb{R}^+, H^{\frac{1}{2}}(\Gamma)) \to H_\sigma^{r}(\mathbb{R}^+, \widetilde{H}^{-\frac{1}{2}}(\Gamma))\ .\end{equation}
\end{description}
\end{theorem}

This theorem is well-known and documented in \cite{bh, costabel04, hd}, as well as \cite{setup} when $\partial\Gamma \neq \emptyset$.
We conclude that the bilinear form $B(\cdot,\cdot)$ is continuous and weakly coercive: 
\begin{proposition}\label{DPbounds} For every $\phi,\psi \in H^1_\sigma( \mathbb{R}^+, \widetilde{H}^{-\frac{1}{2}}(\Gamma))$ there holds:
\begin{equation}|B(\psi,\phi)| \lesssim \|\psi\|_{1,-\frac{1}{2},\Gamma, \ast} \|\phi\|_{1,-\frac{1}{2}, \Gamma,\ast} \quad\textrm{and}\quad\|\psi\|_{0,-\frac{1}{2},\Gamma,\ast}^2 \lesssim B(\psi,\psi) . \end{equation}
\end{proposition}
\begin{proof}
The upper bound follows from Theorem \ref{mapthm}, part a):\\
 $$|B(\psi,\phi)| \leq \|\mathcal{V}\psi\|_{0,\frac{1}{2},\Gamma} \|\partial_t\phi\|_{0,-\frac{1}{2}, \Gamma,\ast}\lesssim \|\psi\|_{1,-\frac{1}{2},\Gamma, \ast} \|\phi\|_{1,-\frac{1}{2}, \Gamma,\ast}.$$
 The lower bound is exactly the assertion in Theorem \ref{mapthm}, part b).
\end{proof}
Note the loss of a time derivative between the upper and lower estimates and in particular that 
\begin{equation}
   \|\psi\|_{0,-\frac{1}{2},\Gamma,\ast}^2 \lesssim B(\psi,\psi) \lesssim \|\psi\|^2_{1,-\frac{1}{2},\Gamma, \ast}\,.
\end{equation}
Alternative inf-sup stable bilinear forms are discussed in \cite{sut}.\\
Then we recall the weak formulation of equation \eqref{weaklysingeq}:
\begin{equation}\label{NP}\small
find \quad\psi \in H^1_\sigma( \mathbb{R}^+, \widetilde{H}^{-\frac{1}{2}}(\Gamma))\quad s.t.\quad B(\psi,\phi)=\langle \partial_t f,\phi\rangle\ \,\,\forall \phi \in H^1_\sigma( \mathbb{R}^+,\widetilde{H}^{-\frac{1}{2}}(\Gamma))\ ,
\end{equation}
as well as its Galerkin discretization in a subspace $V_{\Delta t,\Delta x} \subset H^1_\sigma( \mathbb{R}^+, \widetilde{H}^{-\frac{1}{2}}(\Gamma))$:
\begin{equation}\label{DPdisc}\small
find \,\,\psi_{\Delta t,\Delta x} \in V_{\Delta t,\Delta x}\,\, s.t. \quad B(\psi_{\Delta t,\Delta x},\phi_{\Delta t,\Delta x})=\langle \partial_t f,\phi_{\Delta t,\Delta x}\rangle\ \,\,\forall \phi_{\Delta t,\Delta x} \in V_{\Delta t,\Delta x}\,.
\end{equation}\vspace{-0.5cm}
\begin{corollary} 
Let $f \in H^2_\sigma(\mathbb{R}^+,H^{\frac{1}{2}}(\Gamma))$. The weak formulation \eqref{NP} and its \linebreak Galerkin discretization \eqref{DPdisc}
admit unique solutions $\psi \in H^1_\sigma(\mathbb{R}^+,\widetilde{H}^{-\frac{1}{2}}(\Gamma))$ and \linebreak$\psi_{\Delta t,\Delta x} \in V_{\Delta t,\Delta x}$, respectively.
\end{corollary}

\section{Discretization}\label{Discretization}

We consider discretizations $\mathcal{T}$ of the space-time cylinder $[0,T] \times \Gamma$ by local tensor products in space and time: $[0,T] \times \Gamma= \bigcup_{S_j \in \mathcal{T}} \overline{S_j}$. Here, the $S_j = I_j \times \Gamma_j$ are pairwise disjoint Cartesian products of a time interval $I_j = [\underline{t}_{j}, \overline{t}_j)$ and a segment $\Gamma_j \subset \Gamma$. We write $\Delta t_j=\overline{t}_j-\underline{t}_j$ for the size of the local time step and denote the diameter of $\Gamma_j$ by $\Delta x_j$. Further,  $\Delta t:=\max_j \Delta t_j$, $\Delta x:=\max_j \Delta x_j$. On $I_j$, resp.~$\Gamma_j$, we consider spaces $V_{\Delta t_j}$ and $V_{\Delta x_j}$ of polynomial functions.
The discretization space $V_{\Delta t,\Delta x}$ then consists of functions $\psi$ on $[0,T] \times \Gamma$, such that  $\psi|_{S_j} \in V_{\Delta t_j} \otimes V_{\Delta x_j}$, i.e.~lies in the tensor product of $V_{\Delta t_j}$ and $V_{\Delta x_j}$, for every $S_j$.

In our numerical examples, the {initial} mesh $\mathcal{T}_0$ for the adaptive algorithm {is taken to be a global tensor product} with $N_t N_x$ space-time elements $S_{(n-1)N_x+i}\equiv I_n\times\Gamma_i$, as represented in Figure \ref{Fig:mesh_0}. For the {corresponding} discretization space {$V_{\Delta t,\Delta x}^0$}, we consider the local polynomial degree to be $0$, i.e.,~piecewise constant functions, {both in the space and time variables}. The basis functions of {$V_{\Delta t,\Delta x}^0$} are then given by products of characteristic functions, $\psi_{(n-1)N_x+i}(t,x):=\bar{\psi}_n(t)\bbar{\psi}_i(x)$, supported on a single element $S_{(n-1)N_x+i}$.
$$
\begin{array}{lccccc}
\cline{2-6}
\multicolumn{1}{l|}{I_{N_t}} & \multicolumn{1}{c|}{S_{(N_t-1)N_x+1}} & \multicolumn{1}{c|}{S_{(N_t-1)N_x+2}} & \multicolumn{1}{c|}{S_{(N_t-1)N_x+3}} & \multicolumn{1}{c|}{\ldots} & \multicolumn{1}{c|}{S_{N_tN_x}} \\ \cline{2-6} 
\multicolumn{1}{l|}{\vdots}        & \multicolumn{1}{c|}{\vdots}             & \multicolumn{1}{c|}{\vdots}             & \multicolumn{1}{c|}{\vdots}             & \multicolumn{1}{c|}{} & \multicolumn{1}{c|}{\vdots}               \\ \cline{2-6} 
\multicolumn{1}{l|}{I_3}     & \multicolumn{1}{c|}{S_{2N_x+1}}   & \multicolumn{1}{c|}{S_{2N_x+2}}   & \multicolumn{1}{c|}{\textcolor{blue}{S_{2N_x+3}}}   & \multicolumn{1}{c|}{\ldots} & \multicolumn{1}{c|}{S_{3N_x}}       \\ \cline{2-6} 
\multicolumn{1}{l|}{I_2}     & \multicolumn{1}{c|}{S_{N_x+1}}    & \multicolumn{1}{c|}{S_{N_x+2}}    & \multicolumn{1}{c|}{S_{N_x+3}}    & \multicolumn{1}{c|}{\ldots} & \multicolumn{1}{c|}{S_{2N_x}}       \\ \cline{2-6} 
\multicolumn{1}{l|}{I_1}     & \multicolumn{1}{c|}{S_1}          & \multicolumn{1}{c|}{S_2}          & \multicolumn{1}{c|}{S_3}          & \multicolumn{1}{c|}{\ldots} & \multicolumn{1}{c|}{S_{N_x}}        \\ \cline{2-6}
                             & \multicolumn{1}{c}{\Gamma_1}           & \multicolumn{1}{c}{\Gamma_2}           & \multicolumn{1}{c}{\Gamma_3}           & \multicolumn{1}{c}{\cdots}  & \multicolumn{1}{c}{\Gamma_{N_x}}            
\end{array}
$$
\vspace{-0.5cm}\captionof{figure}{Starting mesh $\mathcal{T}_0$.}\label{Fig:mesh_0} 

Based on an a posteriori error indicator and a marking rule, 
we will mark selected space-time elements (e.g.,~in blue in Figure \ref{Fig:mesh_0}). The marked space-time elements are then refined by halving them both in space  and in time, as depicted in Figure \ref{Fig:mesh_ref}.
$$
\begin{array}{|c|c|cc|c|c|}
\hline
S_{(N_t-1)N_x+1}                & S_{(N_t-1)N_x+2}                & \multicolumn{2}{c|}{S_{(N_t-1)N_x+2}}                        &             \ldots      & S_{N_t N_x}            \\ \hline
               \vdots             &              \vdots               & \multicolumn{2}{c|}{\vdots}                                    &                   &    \vdots                       \\ \hline
\multirow{2}{*}{$S_{2N_x+1}$} & \multirow{2}{*}{$S_{2N_x+2}$} & \multicolumn{1}{c|}{\textcolor{blue} {S_{N_t N_x +2}}} & \textcolor{blue} {S_{N_t N_x+3}} & \multirow{2}{*}{\ldots} & \multirow{2}{*}{$S_{3N_x}$} \\ \cline{3-4}
                            &                             & \multicolumn{1}{c|}{\textcolor{blue} {S_{2N_x+3}}}       & \textcolor{blue} {S_{N_t N_x+1}} &                   &                           \\ \hline
S_{N_x+1}                   & S_{N_x+2}                   & \multicolumn{2}{c|}{S_{N_x+3}}                           &              \ldots     & S_{2N_x}                  \\ \hline
S_1                         & S_2                         & \multicolumn{2}{c|}{S_3}                                 &             \ldots      & S_{N_x}                   \\ \hline
\end{array}
$$
\vspace{-0.5cm}\captionof{figure}{\label{Fig:mesh_ref}Example of mesh refinement.}
\vspace{0.cm}

In this way each marked element $S_j$ is split into 4 smaller space-time elements. We obtain a sequence $\mathcal{T}_k$, $k\geq 1$, of space-time mesh refinements of increasing dimension $N_k = \mathrm{dim} \ V_{\Delta t,\Delta x}^k$. The basis of $V_{\Delta t,\Delta x}^k$ again is given by the characteristic functions $\psi_j(t,x)=\bar{\psi}_j(t)\bbar{\psi}_j(x)$, $j=1,\ldots,N_k$, of $(S_j)_k=(I_j\times\Gamma_j)_k$, i.e.~the $j$-th element of the space-time mesh at refinement level $k$.



The space-time adaptive algorithm is based on the following a posteriori error estimate (stated in \cite{adaptive}), which involves the restriction of the residual $\mathcal{V} (\psi-\psi_{\Delta t,\Delta x}) = :\mathcal{R}$ to the space-time elements $S_j$:

\begin{theorem}\label{apostthm}  Let $\psi  \in H^{1}_\sigma(\mathbb{R}^+,H^{-\frac{1}{2}}(\Gamma))$ be the solution to \eqref{NP}, and $\psi_{\Delta t,\Delta x}$ the solution to \eqref{DPdisc}. 
Then 
\begin{equation}\label{aposteq}\|\psi-\psi_{\Delta t,\Delta x}\|_{0, -\frac{1}{2}, \Gamma, \ast}^2 \lesssim_\sigma  \sum_j\max\{\Delta t_j, \Delta x_j\}\ \left(\|\nabla\mathcal{R}\|_{0,0,  S_j}^2+\|\partial_t\mathcal{R}\|_{0,0,  S_j}^2\right) =:\sum_j \eta_j^2.\end{equation}
\end{theorem}

\vspace{-0.2cm}

The numerical experiments in Section \ref{Numerical results} show the efficiency and robustness of this estimate even for {the} less regular discretization spaces {used here}.

\section{Algorithmic details} \label{algo}

\subsection{Space-time adaptive algorithm} \label{subsec:stadaptive}
The a posteriori error estimate from Theorem \ref{apostthm} leads to an adaptive mesh refinement procedure, based on the steps:

\vspace{-0.5cm}

 \begin{align*}
  \textbf{SOLVE}&\longrightarrow \textbf{ESTIMATE}\longrightarrow \textbf{MARK}\longrightarrow \textbf{REFINE}.
\end{align*}

\vspace{-0.2cm}

The precise algorithm we use is given as follows:\\

\vspace{-0.2cm}

\noindent {\textbf{Space--Time Adaptive Algorithm:}\\
 Input: Datum $f$, mesh $\mathcal{T}_0$, refinement parameter $\Theta  \in (0, 1)$, {exit} tolerance $\epsilon > 0$.
\begin{enumerate}
\item $k=0$.
\item Solve \eqref{DPdisc} on $\mathcal{T}_k$.
\item Compute the local error indicators $\eta_j^2$ in each space-time element $({S_j})_k \in \mathcal{T}_k$ as defined in Theorem \ref{apostthm}.
\item Stop if $\sum_j \eta^2_j < \epsilon$.
\item Find $\eta_{max}^2=\displaystyle\max_{j=1,\ldots,N_k} \eta_j^2$, otherwise.
\item Mark all $({S_j})_k\in \mathcal{T}_k$ with $\eta_j^2 > \Theta \,\eta_{max}^2$.
\item Refine each marked $(S_j)_k$ dividing by 2 both in time and space, as shown in Figure \ref{Fig:mesh_ref}, obtaining in the end a new mesh $\mathcal{T}_{k+1}$. 
\item $k=k+1$.
\item Go to 2.\\
\end{enumerate}

{\bf Remark.} (a) The choice $\Theta=0$ leads to uniform refinements in space and in time.\\
(b) A space-adaptive method based on time-integrated error indicators was considered in \cite{adaptive}. As shown in \cite{graded, mueller}, for polyhedral domains and screens time-independent spatially graded meshes with a sufficiently small uniform time step lead to quasi-optimal convergence rates in spite of the singular behavior of the solution at edges and corners.\\

Finally, since the procedures described in the following subsections hold at each level $k$ of refinement, for the remaining of the Section  we will ignore the index $k$ in the notation.

\subsection{Algebraic reformulation of the discrete problem}

At step 2 of the above written space-time adaptive algorithm, we have to solve \eqref{DPdisc} on the adaptive mesh of the current level of refinement. Taking into account the notation introduced in Section \ref{Discretization}, let us recall that the numerical solution  is written as
\begin{equation}
\label{sol_app}
\psi_{\Delta t,\Delta x}:=\sum_{j=1}^{N} \alpha_j\,\bar{\psi}_j(t)\bbar{\psi}_j(x)\,,\quad
\textrm{with}\quad\psi_{\Delta t,\Delta x\,|S_j}=\alpha_j\,.
\end{equation}
The vector $\pmb{\alpha}$ of coefficients in the linear combination \eqref{sol_app} is obtained as solution to the linear system


\begin{equation}\label{linear_system}
\mathbb{E}\,\pmb{\alpha}=\pmb{\beta}\,,
\end{equation}


whose matrix entries are defined as


\begin{eqnarray} 
\label{matrix_entries}
\hspace{-0.2in}
\mathbb{E}_{i,j}:&=&\int_0^T\int_{\Gamma}\dot{\bar{\psi}}_i(t)\bbar{\psi}_i(x)\int_0^t\int_{\Gamma}\frac{1}{2\pi}\frac{H(t-\tau-|x-y|)}{\sqrt{(t-\tau)^2-|x-y|^2}} \,\bar{\psi}_j(\tau)\bbar{\psi}_j(y)d\gamma_y\,d\tau\,d\gamma_x\,dt\,\nonumber\\
&=&\sum_{\mu,\nu=-1}^0  \int_{\Gamma_i}\bbar{\psi}_i(x)\int_{\Gamma_j}\bbar{\psi}_j(y)\;{\cal G}(t_{i+\mu},t_{j+\nu},|x-y|)\,d\gamma_y\,d\gamma_x\,,
\end{eqnarray}


where


$$
{\cal G}(t,\tau,|x-y|):=\frac{1}{2\pi}H(t-\tau-|x-y|)[\log(t-\tau+\sqrt{(t-\tau)^2-|x-y|^2})-\log(|x-y|)]\,.
$$
The elements of the vector $\pmb{\beta}$ on the right hand side of \eqref{linear_system} are defined as
\begin{equation}
\beta_i:=\int_0^T\int_{\Gamma}\dot{\bar{\psi}}_i(t)\bbar{\psi}_i(x) f(t,x)\,d\gamma_x\,dt=\sum_{\mu=-1}^0 (-1)^{\gamma+1}\int_{\Gamma_j}f(t_{i+\mu},x)\,d\gamma_x.
\end{equation}


Observe that the linear system \eqref{linear_system} loses the typical block lower triangular Toeplitz structure (see e.g. \cite{ADGMP2009}), which is only available for time independent space meshes and uniform time steps.\\
For the accurate evaluation of the weakly singular integrals in the matrix entries \eqref{matrix_entries}, we refer to \cite{ADGMP2009}.\\
To minimize the computational cost, at each refinement step the matrix entries not involved in the refinement are not recomputed. For the sake of clarity, we illustrate the matrix update in case of the refinement depicted in Figure \ref{Fig:mesh_ref}. There, the $(2N_x+3)$-th row and column will be replaced by the entries evaluated on a smaller space-time element $S_ {2N_x+3}$, and three new rows and columns, related to the space-time elements denoted by $S_{N_t N_x+1}, S_{N_t N_x+2}, S_{N_t N_x+3}$, are added to the previous matrix.

\subsection{Implementation of error indicators}

The implementation of the error indicator is based on the evaluation of
$\|\nabla\mathcal{R}\|_{0,0,  {S_j}}^2$ and $\|\partial_t\mathcal{R}\|_{0,0,  {S_j}}^2$, where,
in this setting,


\begin{equation}
    \mathcal{R}(t,x)=f(t,x)-\sum_{i=1}^{N} \alpha_i\int_{\Gamma_i}\int_{t_{i-1}}^{t_i}\frac{1}{2\pi}\frac{H(t-\tau-|x-y|)}{\sqrt{(t-\tau)^2-|x-y|^2}}d\gamma_y d\tau\,.
\end{equation}
For the time derivative $\partial_t\mathcal{R}(t,x)$, we have


\begin{eqnarray*}  
    \partial_t\mathcal{R}(t,x)&=&\partial_t f(t,x)-\sum_{i=1}^{N} \alpha_i\int_{\Gamma_i}\int_{t_{i-1}}^{t_i}-\partial_\tau\left(\frac{1}{2\pi}\frac{H(t-\tau-|x-y|)}{\sqrt{(t-\tau)^2-|x-y|^2}}\right)d\gamma_y d\tau\\
    &=&\partial_t f(t,x)+\frac{1}{2\pi}\sum_{i=1}^{N} \alpha_i\int_{\Gamma_i}\left(\frac{H(t-t_i-|x-y|)}{\sqrt{(t-t_i)^2-|x-y|^2}}-\frac{H(t-t_{i-1}-|x-y|)}{\sqrt{(t-t_{i-1})^2-|x-y|^2}}\right)d\gamma_y\,.
\end{eqnarray*}
In the particular case\footnote{{The more general case of a polygonal boundary $\Gamma$ would require an analysis of the integrals similar to that in \cite{RivPR2011}.}} of $\Gamma=\left\{(x,0),  x\in[0,1]\right\}$, {as considered in the numerical examples,} one has $\Gamma_i:=[x_{i-1},x_i]$ and
\begin{eqnarray*}
\int_{\Gamma_i}\frac{H(t-\tau-|x-y|)}{\sqrt{(t-\tau)^2-|x-y|^2}}dy&=&\int_{0}^{x_i}\frac{H(t-\tau-|x-y|)}{\sqrt{(t-\tau)^2-|x-y|^2}}dy-\int_{0}^{x_{i-1}}\frac{H(t-\tau-|x-y|)}{\sqrt{(t-\tau)^2-|x-y|^2}}dy\\
\int_{0}^{x_i}\frac{H(t-\tau-|x-y|)}{\sqrt{(t-\tau)^2-|x-y|^2}}dy&=&\int^{\min(x,x_i)}_{\max(x-t+\tau,0)}\frac{H(\min(x,x_i)-\max(x-t+\tau,0))}{\sqrt{(t-\tau)^2-|x-y|^2}}dy\\
&+&\int^{\min(t-\tau+x,x_i)}_{x}\frac{H(\min(t-\tau+x,x_i)-x)H(x_i-x)}{\sqrt{(t-\tau)^2-|x-y|^2}}dy\ .
\end{eqnarray*}
Integrating analytically and defining 
\begin{align*}
&    F(t,\tau,x,y):=\\& H(t-\tau)\bigg\{\!\!\!\!\!-\!\!\arctan\left(\frac{x-\min(x,y)}{\sqrt{(t-\tau)^2-(x-\min(x,y))^2}}\right)\vspace{0.2cm} \displaystyle H(\min(x,y)-\max(x-t+\tau,0))\vspace{0.2cm}\\
  &+\! \displaystyle\arctan\left(\frac{x-\max(x-t+\tau,0)}{\sqrt{(t-\tau)^2-(x-\max(x-t+\tau,0))^2}}\right) H(\min(x,y)-\max(x-t+\tau,0))\vspace{0.2cm}\\
&-\!\! \displaystyle\arctan\!\left(\frac{x-\min(t-\tau+x,y)}{\sqrt{(t-\tau)^2-(x-\min(t-\tau+x,y))^2}}\right)
\displaystyle H(\min(t-\tau+x,y)-x)H(y-x)\!\!\bigg\}\!,
\end{align*}
we obtain
\begin{equation*}  
\begin{array}{l}
    \displaystyle\partial_t\mathcal{R}(t,x)=\partial_t f(t,x)\\
    \displaystyle+\frac{1}{2\pi}\sum_{i=1}^{N} \alpha_i\big(F(t,t_i,x,x_i)-F(t,t_i,x,x_{i-1})-F(t,t_{i-1},x,x_i)+F(t,t_{i-1},x,x_{i-1})\big)\ .
\end{array}
\end{equation*}
The norm


\begin{equation*}
    \|\partial_t\mathcal{R}\|_{0,0,  {S_j}}^2=\int_{\Gamma_j}\int_{t_{j-1}}^{t_j}\bigg(\partial_t\mathcal{R}(t,x)\bigg)^2 dx dt 
\end{equation*}
is then computed by Gauss quadrature in space and time.

For the gradient $\nabla\mathcal{R}(t,x)$, in the particular case of $\Gamma=\left\{(x,0), x\in[0,1]\right\}$ {as considered in the numerical examples,} we observe that $\nabla\mathcal{R}(t,x)=\left(\partial_x\mathcal{R}(t,x),0\right)$,  with 
\begin{eqnarray*}  
    \partial_x\mathcal{R}(t,x)&=&\partial_x f(t,x)-\sum_{i=1}^{N} \alpha_i\int_{\Gamma_i}\int_{t_{i-1}}^{t_i}-\partial_y\left(\frac{1}{2\pi}\frac{H(t-\tau-|x-y|)}{\sqrt{(t-\tau)^2-|x-y|^2}}\right)dy d\tau\\
    &=&\partial_x f(t,x)+\frac{1}{2\pi}\sum_{i=1}^{N} \alpha_i\int_{t_{i-1}}^{t_i}\left(\frac{H(t-\tau-|x-x_i|)}{\sqrt{(t-\tau)^2-|x-x_i|^2}}-\frac{H(t-\tau-|x-x_{i-1}|)}{\sqrt{(t-\tau)^2-|x-x_{i-1}|^2}}\right)d\tau\,.
\end{eqnarray*}
Using that
$$
\int_{0}^{t_i}\frac{H(t-\tau-|x-y|)}{\sqrt{(t-\tau)^2-|x-y|^2}}d\tau=\int_0^{\min(t_i,t-|x-y|)}\frac{H(\min(t_i,t-|x-y|))}{\sqrt{(t-\tau)^2-|x-y|^2}}d\tau\,
$$
and defining
\begin{eqnarray*}
    S(t,\tau,x,y)&:=&\left(-\log\left(t-\min(\tau,t-|x-y|)+\sqrt{(t-\min(\tau,t-|x-y|))^2-(x-y)^2}\right)\right.\\
    &+&\left.\log\left(t+\sqrt{t^2-(x-y)^2}\right)\right)H(\min(\tau,t-|x-y|))\,,
\end{eqnarray*}
we obtain
\begin{equation*}
\begin{array}{l}
      \displaystyle\partial_x\mathcal{R}(t,x)=\partial_x f(t,x)\\
      \displaystyle +\frac{1}{2\pi}\sum_{i=1}^{N} \alpha_i\big(S(t,t_i,x,x_i)-S(t,t_{i-1},x,x_{i})-S(t,t_i,x,x_{i-1})+S(t,t_{i-1},x,x_{i-1})\big)\,.
\end{array}
\end{equation*}


The norm


\begin{equation*}
    \|\partial_x\mathcal{R}\|_{0,0,  {S_j}}^2=\int_{\Gamma_j}\int_{t_{j-1}}^{t_j}\bigg(\partial_x\mathcal{R}(t,x)\bigg)^2 dx dt 
\end{equation*}


is then computed by Gauss quadrature in space and time.

\section{Numerical results}\label{Numerical results}
In the following numerical experiments we solve \eqref{DPdisc} on $\Gamma=\left\{(x,0), x\in[0,1]\right\}$ for $t\in[0,T]=[0,1]$, using both adaptive and uniform discretizations. The Dirichlet datum $f$ is specified in the respective examples. Unless stated otherwise, the adaptive algorithm is started with a coarse initial mesh ${\cal T}_0$ using $\Delta t=\Delta x=0.25$, i.e., with $N_x=N_t=4$. The coarsest uniform space-time discretization parameters are given by $\Delta t=\Delta x=0.1$, i.e., $N_x=N_t=10$. The adaptive numerical experiments are then based on the space-time adaptive algorithm from Subsection \ref{subsec:stadaptive}. In addition to the theoretical local error indicators from Theorem \ref{apostthm}, i.e., $\eta_j^2:=\max\{\Delta t_j, \Delta x_j\}\ \left(\|\nabla\mathcal{R}\|_{0,0,  {S_j}}^2+\|\partial_t\mathcal{R}\|_{0,0,  {S_j}}^2\right)$, we also use heuristic local error indicators defined by


\begin{equation}\label{Euristic}
    \tilde{\eta}_j^2:= \Delta x_j\ \|\nabla\mathcal{R}\|_{0,0,  {S_j}}^2+\|\partial_t\mathcal{R}\|_{0,0,  {S_j}}^2\,.
\end{equation}


The presented errors are related to the discrete energy defined by 
\begin{equation}\label{energia}
   {E_{\Delta t,\Delta x}} := B(\psi_{\Delta t,\Delta x},\psi_{\Delta t,\Delta x})=\iint_{[0,T] \times \Gamma} \partial_t f(t,x)\ \psi_{\Delta t,\Delta x}(t,x)\ d\gamma_x \, dt\,. 
\end{equation}


The squared errors in energy norm will be plotted with respect to the total number of space-time $DoF\!s$, together with the corresponding error indicators $\displaystyle \sum_j \eta_j^2$ and $\displaystyle \sum_j \tilde{\eta}_j^2$.
\\
In every example the local error indicators are computed by a Gaussian quadrature rule with $16\times 16$ nodes in each cell of the space-time discretization, to achieve the necessary accuracy.

\subsection{Example with smooth benchmark solution}\label{example1}
We solve \eqref{DPdisc} for the Dirichlet datum $f$ which corresponds to the smooth exact solution $\psi(t,x)=xt$, shown {on the left of} Figure \ref{Fig:xt}.
We consider five uniform space-time uniform meshes, obtained from the initial space-time mesh by halving the space and time steps. In this way $\Delta t=\Delta x=1/N_x=1/N_t$.\\
Denoting by $E_{\Delta t, \Delta x}^{(i)}$ the discrete energy obtained for $\Delta x=\Delta t=1/(10\cdot 2^i)$, 
we find the empirical rate of convergence 
for $i=1,2,3$,
$$
\frac{E_{\Delta t, \Delta x}^{(i+1)}-E_{\Delta t, \Delta x}^{(i)}}{E_{\Delta t, \Delta x}^{(i)}-E_{\Delta t, \Delta x}^{(i-1)}}\simeq 4.115
$$
and infer a benchmark value $E\approx 0.37135e-01$ for the exact energy.

For the  space-time adaptive algorithm the refinement parameter in step 6.~is here chosen as $\Theta=0.2$.

We first consider the decay of the squared $L^2$ error 
$\|\psi-\psi_{\Delta t,\Delta x}\|_{0,0}^2$ and of the squared energy error for uniform refinements in terms of $DoF\!s=N_x N_t$, as depicted in Figure \ref{Example1_1}. 
These slopes correspond to an $O(\Delta x)$ decay for the squared $L^2$ error, an $O(\Delta x^2)$ 
\begin{minipage}{13.cm}
\hspace{0.in}\includegraphics[trim=0 0 0 0, clip,scale=0.33]{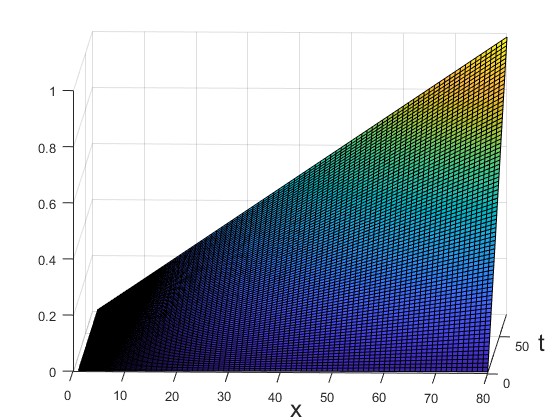}\hspace{-0.cm}
\includegraphics[trim=20 0 0 0, clip,scale=0.44]{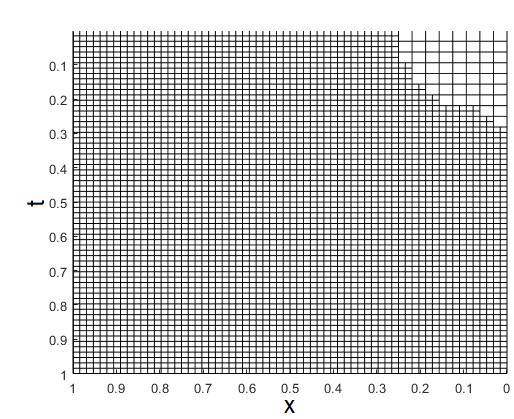}
\vspace{-0.2cm}\captionof{figure}{{Example \ref{example1}: exact solution (left) and refined mesh (right) using the theoretical error indicator and $\Theta=0.2$.}}\vspace{-0.2cm}
\label{Fig:xt}
\end{minipage}

decay for the energy error, in agreement with the heuristic indicator, and an $O(\Delta x^3)$ decay for the theoretical error indicator.
Hence, we note that the squared error in energy norm, like the heuristic error indicator, behaves as $\Delta x \|\psi-\psi_{\Delta t,\Delta x}\|_{0,0}^2\simeq \|\psi-\psi_{\Delta t,\Delta x}\|_{\frac{1}{2}, -\frac{1}{2}, \Gamma, \ast}^2$, while the theoretical error indicator behaves as $\Delta x \Delta t \|\psi-\psi_{\Delta t,\Delta x}\|_{0,0}^2$. Both are greater than or equal to $\|\psi-\psi_{\Delta t,\Delta x}\|_{0, -\frac{1}{2}, \Gamma, \ast}^2$, {and the theoretical error indicator estimates a weaker norm than the energy norm}, in agreement with the theoretical analysis.\\
In Figure \ref{Example1_2}, we show the decay of the energy error in terms of $DoF\!s$ for adaptive space-time refinements, driven respectively by theoretical and heuristic error indicators. The slopes for both energy errors are similar and follow the behavior of the heuristic error indicator, while the theoretical error indicator decays faster, as explained above.\\
Uniform refinements are quasi-optimal for the smooth solutions like in the current example. Figure \ref{Example1_3} accordingly shows the same convergence rate $O(DoF\!s^{-1})$ of the squared energy errors obtained for the uniform and adaptive refinements. At a fixed number of $DoF\!s$ the error obtained with a uniform mesh is lower. {The resulting adaptive refinements are almost uniform in spite of the low value of $\Theta$, as shown in Figure \ref{Fig:xt} (right) for refinements based on the theoretical error indicator.} \\

\begin{center}
\begin{minipage}{13.cm}
\hspace{0.5in}\includegraphics[trim=0 0 0 0, clip,scale=0.37]{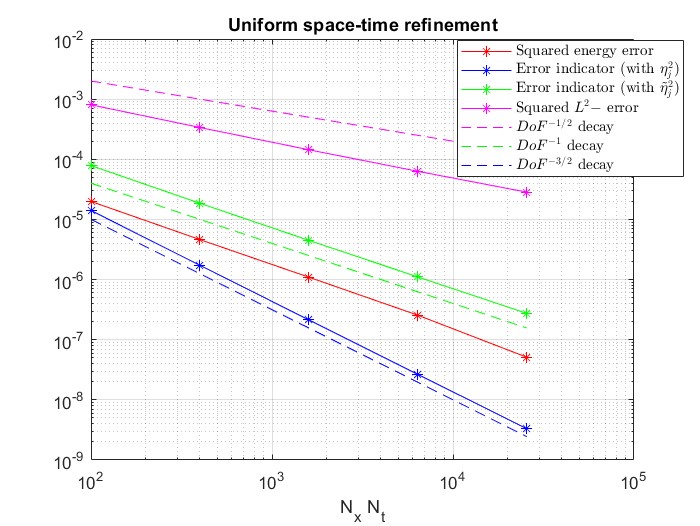}
\vspace{-0.2cm}\caption{Example \ref{example1}: decay of squared $L^2$  and energy errors w.r.t.~degrees of freedom $DoF\!s=N_xN_t$.}
\label{Example1_1}
\end{minipage}
\end{center}
\begin{center}
\begin{minipage}{13.cm}
\hspace{0.5in}\includegraphics[trim=0 0 0 0, clip,scale=0.37]{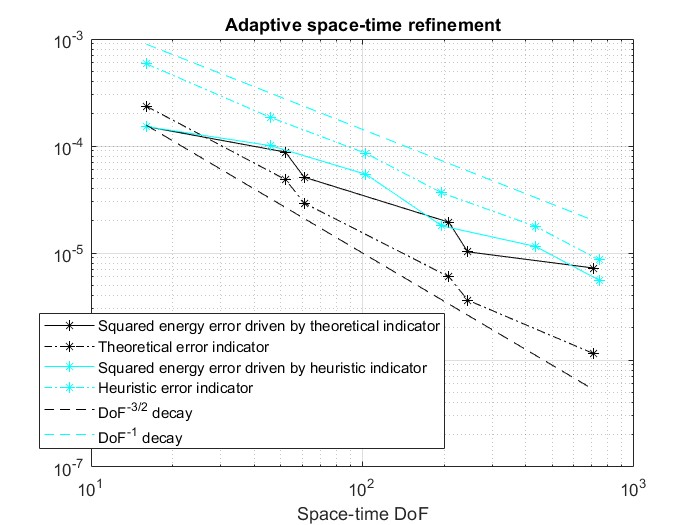}
\vspace{-0.2cm}\caption{Example \ref{example1}: decay of the energy errors w.r.t.~the space-time $DoF\!s$, driven respectively by theoretical and heuristic error indicators.}
\label{Example1_2}
\end{minipage}
\end{center}

\subsection{Example with peak-shaped datum}\label{example2}
We solve \eqref{DPdisc} for the Dirichlet datum $f$ shown in Figure \ref{Fig:Pyramid} on the left, which features a peak and has compact support in a small part of $[0,T]\times\Gamma$:
$$\small
f(t,x)=H(t)\left\{\begin{array}{ll}
     \sin^4(4\pi t)\sin^4\big(3\pi(x-1)\big)H(x-\frac{1}{3})H(\frac{2}{3}-x)&\textrm{for}\quad 0\leq t\leq \frac{1}{8}\\
     \sin^4(4\pi (-t-\frac{2}{8})\sin^4\big(3\pi(x-1)\big)H(x-\frac{1}{3})H(\frac{2}{3}-x)&\textrm{for} \quad \frac{1}{8}\leq t\leq \frac{2}{8}\\
     0 &\textrm{for}\quad t\geq \frac{2}{8}
\end{array}\right.
$$
The numerical solution  is shown in Figure \ref{Fig:Pyramid} on the right: it features two positive peaks and a negative peak, as well as waves traveling outwards from these peaks in space and time. We expect adaptive mesh refinements along the space-time support of the solution.
\begin{center}
\begin{minipage}{13.cm}
\hspace{0.5in}\includegraphics[trim=0 0 0 0, clip,scale=0.37]{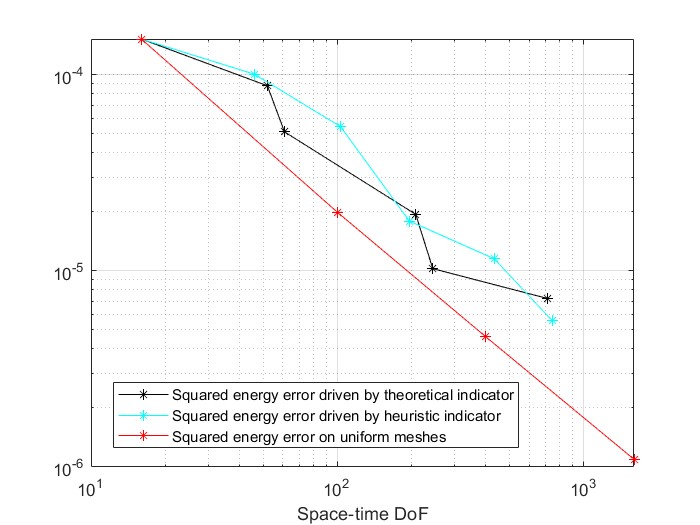}
\vspace{-0.2cm}\caption{Example \ref{example1}: comparison between the squared energy errors  for uniform and adaptive refinements.}\vspace{-0.2cm}
\label{Example1_3}
\end{minipage}
\end{center}
\begin{center}
\begin{minipage}{13.cm}
\hspace{0.in}\begin{center}\vspace{-0.3cm}
$\begin{array}{lc}
   \hspace{-0.5cm}\begin{minipage}{6.5cm}
\hspace{0.0cm}\includegraphics[trim=5 0 30 0, clip,scale=0.43]{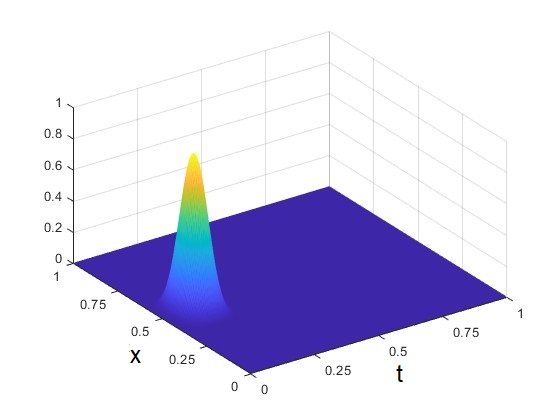}
\end{minipage}  & \hspace{-0.2cm}\begin{minipage}{6.cm}
\hspace{0.0cm}\includegraphics[trim=30 0 30 0, clip,scale=0.44]{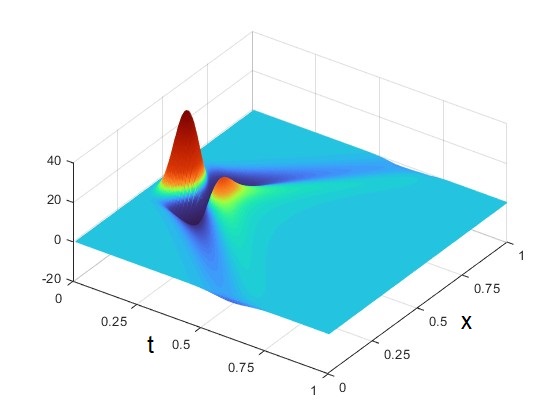}
\end{minipage}
\end{array}$\captionsetup{type=figure}
\vspace{-0.2cm}\captionof{figure}{Example \ref{example2}: on the left, peak-shaped Dirichlet datum (left), numerical solution  $\psi_{\Delta t, \Delta x}$ obtained with a uniform mesh of $N_x=N_t=160$ space-time elements (right).}\label{Fig:Pyramid}
\end{center}
\end{minipage}
\end{center}
As in the previous example, we consider five uniform  space-time uniform meshes, obtained from the initial space-time mesh by halving the space and time steps. In this way $\Delta t=\Delta x=1/N_x=1/N_t$. 
Denoting by $E_{\Delta t, \Delta x}^{(i)}$ the discrete energy obtained for $\Delta x=\Delta t=1/(10\cdot 2^i)$, 
we find the empirical rate of convergence 
for $i=1,2,3$,
$$
\frac{E_{\Delta t, \Delta x}^{(i+1)}-E_{\Delta t, \Delta x}^{(i)}}{E_{\Delta t, \Delta x}^{(i)}-E_{\Delta t, \Delta x}^{(i-1)}}\simeq 3.89
$$
and infer a benchmark value $E\approx 3.57403e+01$ for the exact energy.\\
For the  space-time adaptive algorithm the refinement parameter in step 6.~is here chosen as $\Theta=0.5$. \\
We first consider the decay of the squared energy error for uniform refinements in terms of $DoF\!s=N_x N_t$, as depicted in Figure \ref{Example2_1}. 
As before, these slopes correspond to an $O(\Delta x^2)$ decay for the energy error, in agreement with the heuristic indicator, and an $O(\Delta x^3)$ decay for the theoretical error indicator.

In Figure \ref{Example2_2} we show  the decay of the energy error in terms of $DoF\!s$  for adaptive space-time refinements, driven by the theoretical, respectively  heuristic, error indicators.

\begin{minipage}{13.cm}
\hspace{0.5in}\includegraphics[trim=0 0 0 0, clip,scale=0.37]{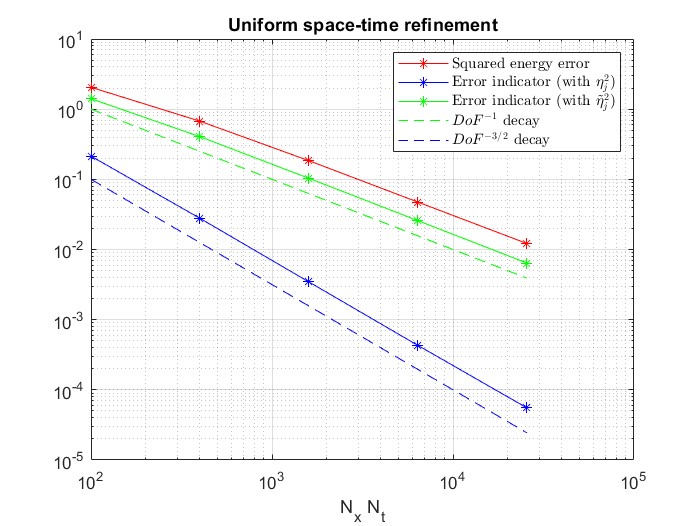}
\captionof{figure}{Example \ref{example2}: decay of squared $L^2$  and energy errors w.r.t.~degrees of freedom $DoF\!s=N_xN_t$.}
\label{Example2_1}
\end{minipage}
\begin{minipage}{13.cm}
\hspace{0.5in}\includegraphics[trim=0 0 0 0, clip,scale=0.37]{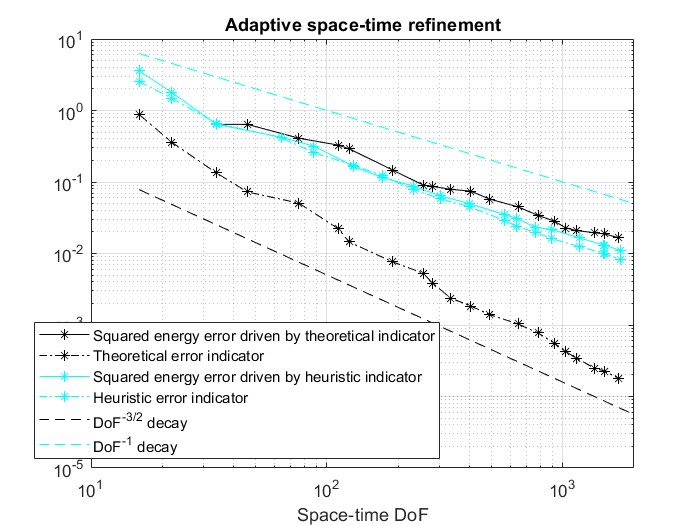}
\captionof{figure}{Example \ref{example2}: decay of the energy errors w.r.t.~the space-time $DoF\!s$, driven respectively by theoretical and heuristic error indicators.}
\label{Example2_2}
\end{minipage}

\vspace*{0.4cm}

It is worth noting that, as in Example \ref{example1}, the convergence of both energy errors is similar and follows the behavior of the heuristic error indicator. {Analogous to Example \ref{example1}, the theoretical error indicator decays faster, as analytically expected, estimating a weaker norm.}

In Figure \ref{Example2_3} we depict an adaptive mesh driven by the theoretical error indicator (left) and the corresponding numerical solution (right) obtained by the adaptive algorithm with {exit} tolerance $\epsilon=10^{-5}$.

Figure \ref{Example2_3_bis} presents meshes with similar numbers of $DoF\!s$, obtained using the space-time adaptive algorithm driven respectively by theoretical and heuristic error indicators: both show similar refinements in accordance with the expectations.

Finally, Figure \ref{Example2_4} (left) shows a comparison between the squared energy errors obtained by the uniform and adaptive refinements: the slopes for both energy errors are similar, in line with $O(DoF\!s^{-1})$, as the solution in this example is still smooth. Unlike in Example \ref{example1}, however, because of the features of the solution, at a fixed number of $DoF\!s$ the error obtained from adaptive refinements is significantly lower.\\
Figure \ref{Example2_4} (right) shows the squared energy error decay w.r.t.~memory consumption: for the error levels considered, the adaptive refinements also save memory compared to uniform refinements, even though the space-time matrix no longer has a Toeplitz structure. Advantages in memory can only be expected when the adaptive approach leads to a significant reduction of $DoF\!s$, which overcomes the savings from the Toeplitz structure for uniform meshes, as also in Example \ref{example4} or problems in $\mathbb{R}^3$.
\begin{minipage}{13.cm}
\hspace{0.in}\includegraphics[trim=0 0 0 0, clip,scale=0.46]{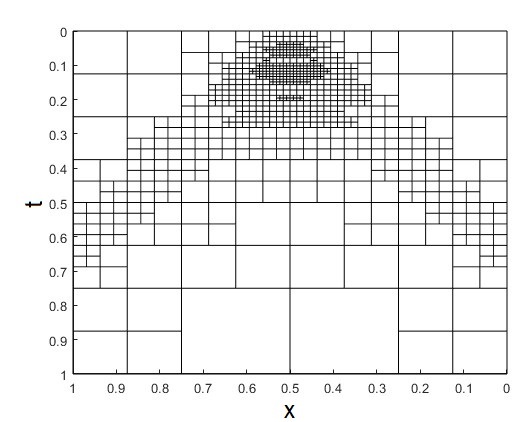}\hspace{-0.2cm}
\includegraphics[trim=0 0 0 0, clip,scale=0.46]{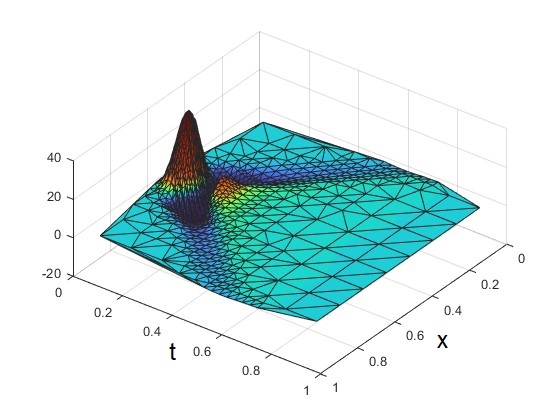}
\vspace{-0.2cm}\captionof{figure}{Example \ref{example2}: final mesh driven by theoretical error indicator, obtained by the adaptive algorithm with {exit} tolerance $\epsilon=10^{-5}$ (left), {corresponding to $15$ steps of the Space–Time Adaptive Algorithm, and} the numerical solution  $\psi_{\Delta t, \Delta x}$ (right).}\vspace{-0.2cm}
\label{Example2_3}
\end{minipage}
\begin{minipage}{13.cm}

     \includegraphics[trim=0 0 0 0, clip,scale=0.46]{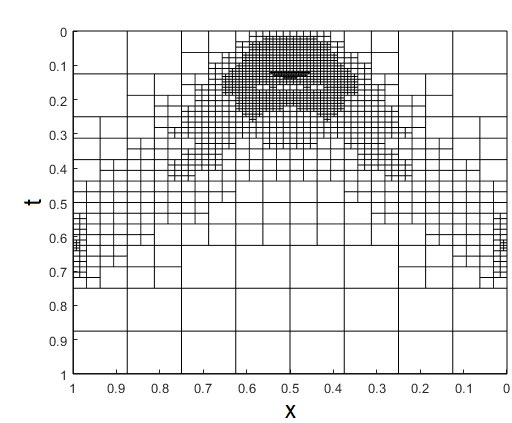}\hspace{0.1cm}\includegraphics[trim=15 0 0 0, clip,scale=0.46]{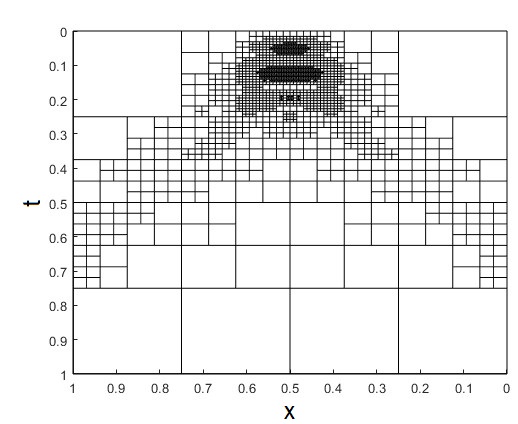} 
\hspace{-0.5cm}
\captionof{figure}{Example \ref{example2}: meshes having similar number of $DoF\!s$, obtained using adaptive algorithms driven respectively by theoretical (left) and heuristic (right) error indicators {with 24 iterations of the Space–Time Adaptive Algorithm}.}
\label{Example2_3_bis}
\end{minipage}

\begin{minipage}{13.cm}
\hspace{-0.7cm}
$\begin{array}{rl}
\begin{minipage}{6.4cm}
\includegraphics[trim=0 0 0 0, clip,scale=0.29]{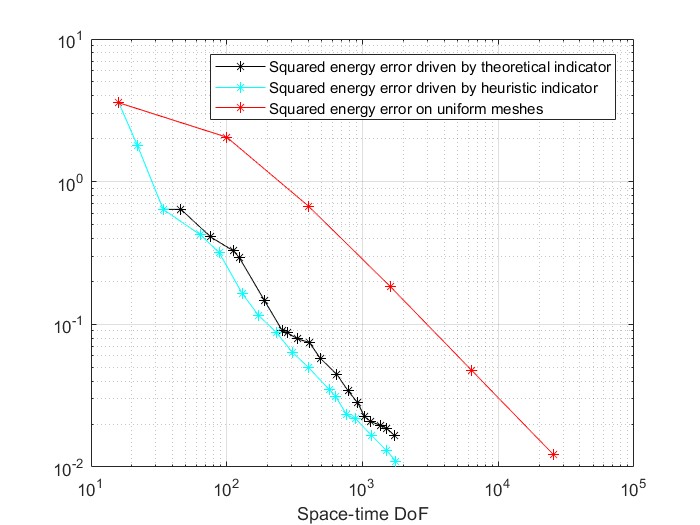}
\end{minipage}
&\hspace{-0.2cm}
\begin{minipage}{6.4cm}
\includegraphics[trim=20 0 44 0, clip,scale=0.36]{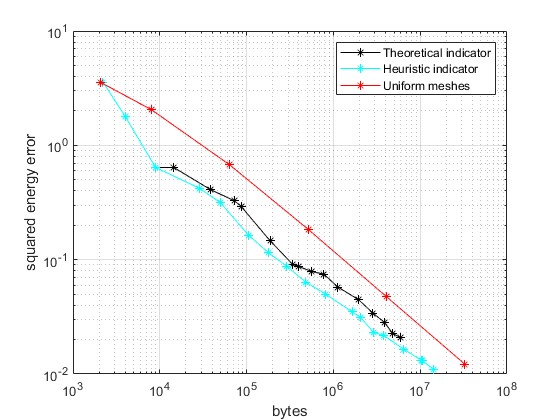}
\end{minipage}
\end{array}$
\vspace{-0.2cm}
\captionof{figure}{Example \ref{example2}: comparison between the squared energy errors  for uniform and adaptive refinements w.r.t.~$DoF\!s$ (left) and w.r.t.~memory consumption (right).}
\label{Example2_4}\vspace{-0.2cm}
\end{minipage}
\subsection{Example with  solution singular at the endpoints}\label{example3}
We solve \eqref{DPdisc} for the Dirichlet datum $f$ considered in \cite{ADGMP2009}, given by


$$
f(t,x)=H(t-kx)\left\{\begin{array}{ll}
     \sin^2\big(4\pi(t-kx)\big)&\textrm{for}\quad 0\leq t-kx\leq \frac{1}{8},\\
     1&\textrm{for}\quad t-kx\geq \frac{1}{8}, 
\end{array}\right. 
$$


with $k=\cos(\theta)$ and $\theta\in(0,\pi)$. To be specific, we fix 
$\theta=\pi/2$. 

The numerical solution  is shown in Figure \ref{Example3_0}: near the endpoints of $\Gamma$ it features the typical geometric singularity, tending to $\infty$ like the inverse of the square root of the distance to the closest endpoint. The singular behavior is further illustrated by the snapshot of the solution at fixed time $t=1$ in Figure \ref{Example3_0_bis} on the right. The evolution in time in the points $x=0.4875, 0.9875$ is shown in the same figure on the left, which suggests that adaptive mesh refinements may be expected near $t=0$, in addition to 
the endpoints of $\Gamma$.

Here we consider six uniform  space-time uniform meshes, obtained from the initial space-time mesh by halving the space and time steps. In this way $\Delta t=\Delta x=1/N_x= 1/N_t$. Denoting by $E_{\Delta t, \Delta x}^{(i)}$ the discrete energy obtained for $\Delta x=\Delta t=1/(10\cdot 2^i)$, we find the empirical rate of convergence 
for $i=1,2,3$,
$$
\frac{E_{\Delta t, \Delta x}^{(i+1)}-E_{\Delta t, \Delta x}^{(i)}}{E_{\Delta t, \Delta x}^{(i)}-E_{\Delta t, \Delta x}^{(i-1)}}\simeq 3.3372
$$
and infer a benchmark value 
$E\approx 2.07339e+01$ for the exact energy. \\
\begin{minipage}{13.cm}
\hspace{3cm}
\includegraphics[trim=0 0 0 0, clip,scale=0.43]{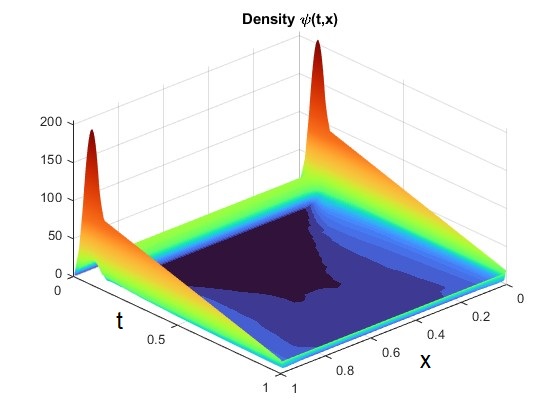}\vspace{-0.2cm}
\captionof{figure}{Example \ref{example3}: evolution of numerical solution $\psi_{\Delta t, \Delta x}$ in space and time}\vspace{-0.cm}
\label{Example3_0}
\end{minipage}



For the  space-time adaptive algorithm the refinement parameter in step 6.~is here chosen as $\Theta=0.5$.\\

We first consider the decay of the squared energy error for uniform refinements in terms of $DoF\!s=N_x N_t$, as depicted in Figure \ref{Example3_1}.  The theoretical error indicator decays as $O(\Delta x)$, as expected for a solution with square-root singularity at the crack tips \cite{graded}. The convergence of the squared energy error and  the heuristic error indicator is still in a pre-asymptotic regime and, for a growing number of $DoF\!s$, is expected to approach the slower rate of convergence $O(\Delta x)$.
In fact, as the mesh is the same for both error indicators, and since $\eta_j^2\leq \tilde{\eta}_j^2$ the green line cannot intersect the blue line but can, at most, approach it if {$\|\partial_t R\|$ is negligible w.r.t.~$\|\nabla R\|$.} As the heuristic error indicator is consistent with the squared energy error, asymptotically the red curve is expected to become parallel to the blue one.

\begin{minipage}{13.cm}
\hspace{0.cm}
\includegraphics[trim=0 0 0 0, clip,scale=0.34]{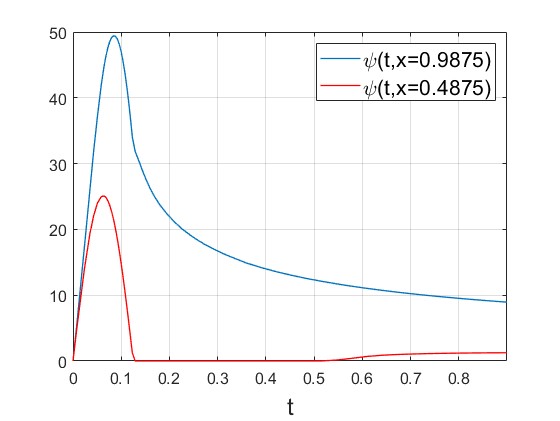}\hspace{-0.cm}\includegraphics[trim=30 0 50 0, clip,scale=0.34]{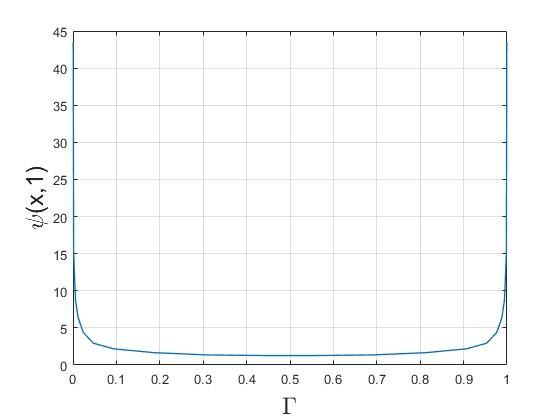}\vspace{-0.2cm}
\captionof{figure}{Example \ref{example3}: time profiles of  numerical solution $\psi_{\Delta t, \Delta x}$ at $x=0.4875, 0.9875$ (left), space profile at $t=1$ (right).}\vspace{-0.2cm}
\label{Example3_0_bis}
\end{minipage}

\begin{minipage}{13.cm}
\hspace{0.5in}\includegraphics[trim=0 0 0 0, clip,scale=0.37]
{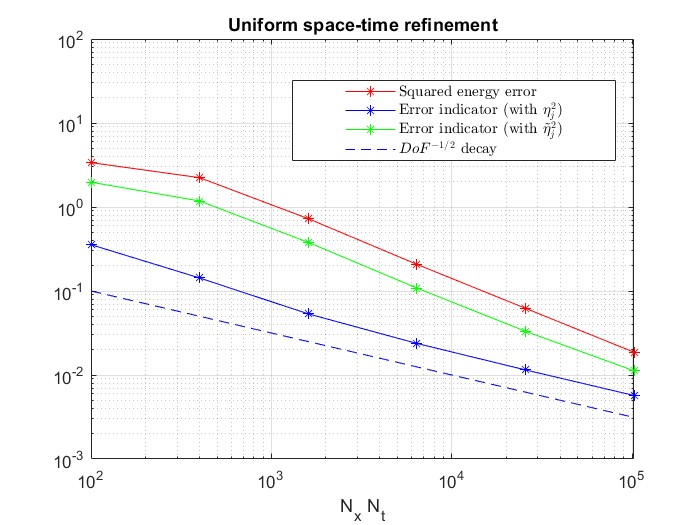}
\vspace{-0.2cm}\captionof{figure}{Example \ref{example3}: decay of squared $L^2$  and energy errors w.r.t.~degrees of freedom $DoF\!s=N_xN_t$.}\vspace{-0.2cm}
\label{Example3_1}
\end{minipage}
\begin{minipage}{13.cm}
\hspace{0.5in}\includegraphics[trim=0 0 0 0, clip,scale=0.37]{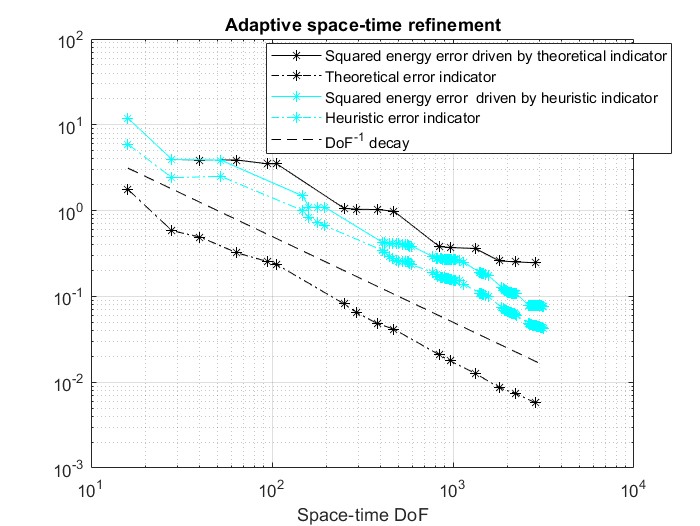}
\captionof{figure}{Example \ref{example3}: decay of the energy errors w.r.t.~the space-time $DoF\!s$, driven respectively by theoretical and heuristic error indicators.}
\label{Example3_2}
\end{minipage}

\vspace*{0.4cm}

In Figure \ref{Example3_2} we show the decay of the energy error in terms of $DoF\!s$  for adaptive space-time refinements, driven by the theoretical, respectively  heuristic, error indicators. Unlike in the previous experiments, we find the same convergence rates for both energy errors and their indicators in this problem, {indicating the dominant effect of the spatial singularity.}
Figure \ref{Example3_4} shows a comparison between the squared energy errors obtained by the uniform and adaptive refinements: the slopes for the adaptive approaches, in line with $O(DoF\!s^{-1})$, is about twice the rate $O(DoF^{-1/2})$ of the uniform approach, suggesting the expected convergence \cite{graded}.

In Figure \ref{Example3_3} we finally depict  adaptive meshes driven by the theoretical error indicator, for different initial meshes ${\cal T}_0$ and different {exit}
tolerances. Both show similar refinements in accordance with the expectations.

\begin{minipage}{13.cm}
\hspace{0.5in}\includegraphics[trim=0 0 0 0, clip,scale=0.37]{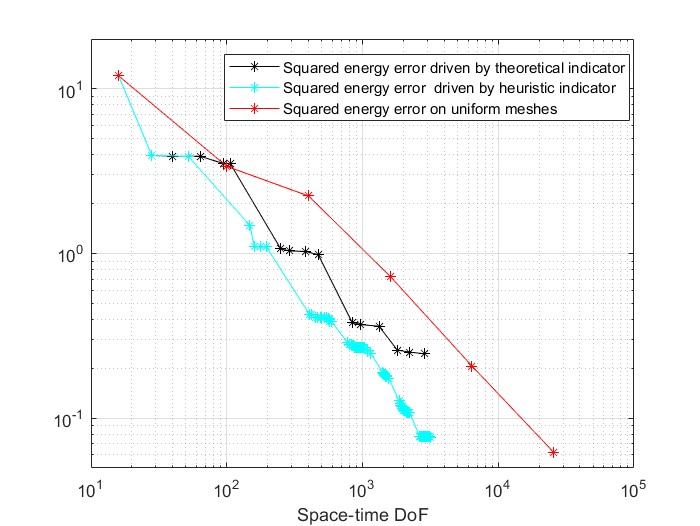}
\captionof{figure}{Example \ref{example3}: comparison between the squared energy errors for uniform and adaptive refinements.}\vspace{-0.2cm}
\label{Example3_4}
\end{minipage}
\begin{minipage}{13.cm}
\hspace{0.in}\begin{center}
$\begin{array}{rl}
   \hspace{-0.2cm}\begin{minipage}{6.5cm}
\includegraphics[trim=0 0 0 0, clip,scale=0.44]{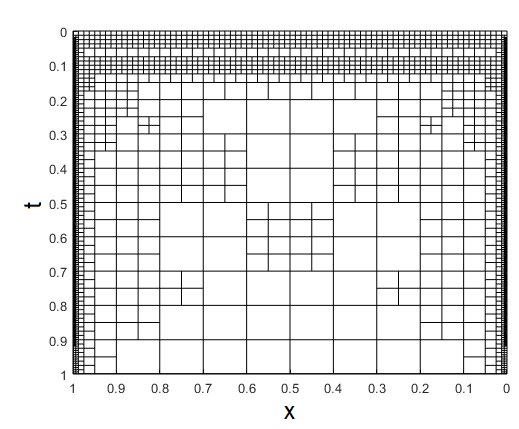}\end{minipage} &\hspace{-0.5cm}\begin{minipage}{6.5cm}
\includegraphics[trim=10 0 0 0, clip,scale=0.36]{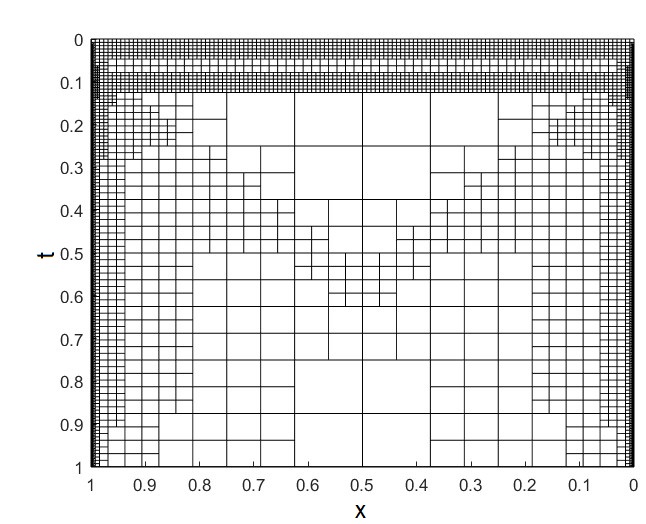}\end{minipage}
\end{array}$
\vspace{-0.2cm}\captionof{figure}{Example \ref{example3}: adaptive meshes driven by theoretical error indicator. On the left: ${\cal T}_0$ with $\Delta x=\Delta t=0.1$ and $\epsilon=10^{-5}$,  on the right: ${\cal T}_0$ with $\Delta x=\Delta t=0.5$ and $\epsilon=10^{-6}$}\vspace{-0.2cm}
\label{Example3_3}
\end{center}
\end{minipage}


\subsection{Example with solution singular in space and in time}\label{example4}
We solve \eqref{DPdisc} for the Dirichlet datum $f$ given by
$f(t,x)=H(t)t^{2/3}$.
As shown in Figure \ref{Example4_0}, the numerical solution contains the geometric singularity at the endpoints of $\Gamma$, as in Example \ref{example3}, as well as a mild time singularity at $t=0$. Adaptive mesh refinements may therefore be expected near $t=0$, in addition to the endpoints of $\Gamma$. 

As in the first two examples, we consider five uniform  space-time uniform meshes, obtained from the initial space-time mesh by halving the space and time steps. In this way $\Delta t=\Delta x=1/N_x=1/N_t$.
Denoting by $E_{\Delta t, \Delta x}^{(i)}$ the discrete energy obtained for $\Delta x=\Delta t=1/(10\cdot 2^i)$, 
we find the empirical rate of convergence 
for $i=1,2,3$,
$$
\frac{E_{\Delta t, \Delta x}^{(i+1)}-E_{\Delta t, \Delta x}^{(i)}}{E_{\Delta t, \Delta x}^{(i)}-E_{\Delta t, \Delta x}^{(i-1)}}\simeq 1.317
$$
and infer a benchmark value 
$E\approx 3.64917$ for the exact energy.\\ 
\begin{minipage}{12.cm}
\hspace{0.in}
\hspace{3cm}
\includegraphics[trim=40 0 0 0, clip,scale=0.43]{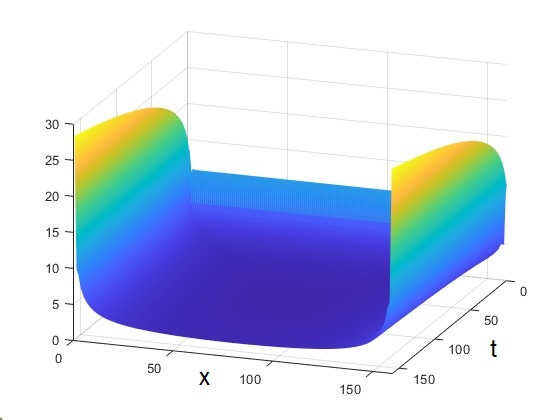}
\captionof{figure}{Example \ref{example4}: evolution of the numerical solution $\psi_{\Delta t, \Delta x}$ in space and time.}\label{Example4_0}
\end{minipage}

\vspace*{0.4cm}

For the  space-time adaptive algorithm the refinement parameter in step 6.~is here chosen as $\Theta=0.5$.

We first consider the decay of the squared energy error for uniform refinements in terms of $DoF\!s=N_x N_t$, as depicted in Figure \ref{Example4_1}.  The theoretical error indicator converges like $O(\Delta x)$, as expected for a solution with square-root singularity at the crack tips. The squared energy error decays like $O(\Delta x^{1/2})$, in agreement with the heuristic error indicator. 
In Figure \ref{Example4_2} we show  the decay of the energy error in terms of $DoF\!s$  for adaptive space-time refinements driven by the  theoretical, respectively heuristic, error indicators. It is worth noting that, as in Examples \ref{example1} {and \ref{example2}}, the slopes for both energy errors
are similar and follow the behavior of the heuristic 
indicator, while the theoretical 
indicator decays faster, as explained in Example \ref{example1}.

Figure \ref{Example4_3}  presents meshes with similar numbers of $DoF\!s$, obtained using the space-time adaptive algorithm driven respectively by theoretical and heuristic error indicators: both show similar refinements in accordance with the expectations.

\begin{minipage}{13.cm}
\hspace{2.cm}
\includegraphics[trim=40 0 0 0, clip,scale=0.36]{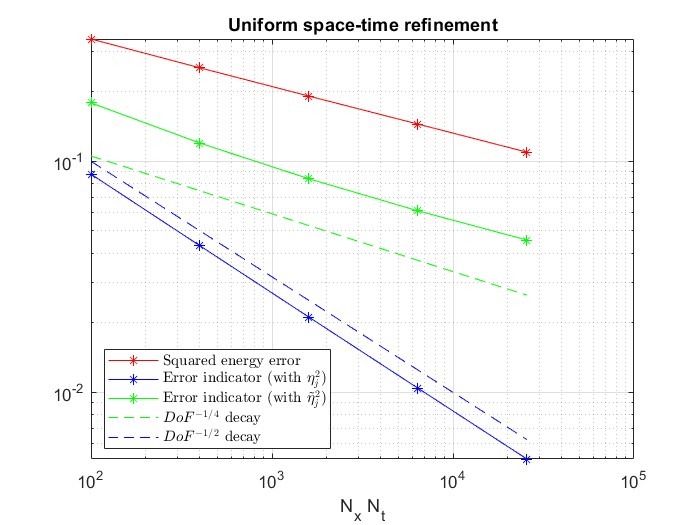}
\vspace{-0.2cm}\captionof{figure}{Example \ref{example4}: decay of squared $L^2$  and energy errors w.r.t.~degrees of freedom $DoF\!s=N_xN_t$.}\vspace{-0.2cm}
\label{Example4_1}
\end{minipage}
\begin{minipage}{13.cm}
\hspace{2.cm}
\includegraphics[trim=40 0 0 0, clip,scale=0.36]{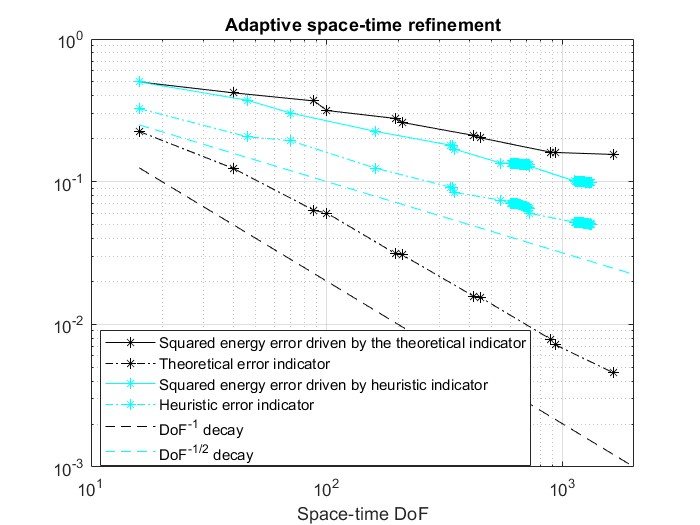}
\captionof{figure}{Example \ref{example4}: decay of the energy errors w.r.t.~the space-time $DoF\!s$, driven respectively by theoretical and heuristic error indicators.}\label{Example4_2}
\end{minipage}
\begin{minipage}{13.cm}
\hspace{-.2cm}\includegraphics[trim=0 0 0 0, clip,scale=0.45]{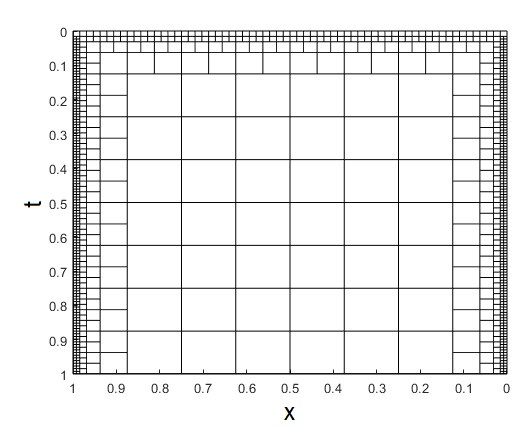}\hspace{0.5cm}
\includegraphics[trim=20 0 0 0, clip,scale=0.45]{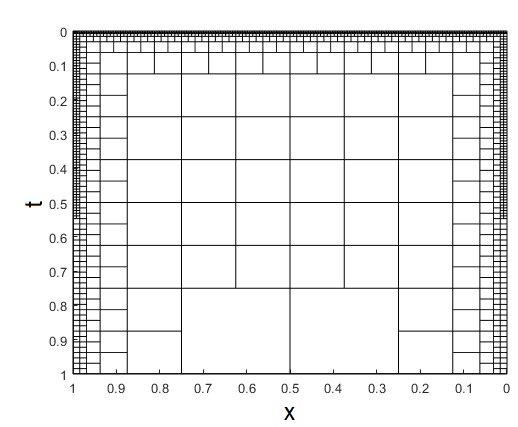}
\captionof{figure}{Example \ref{example4}: last mesh of the adaptive algorithm driven by theoretical error indicator (left), last mesh of the adaptive algorithm driven by heuristic error indicator (right)}
\label{Example4_3}
\end{minipage}

\vspace*{0.4cm}

Finally, Figure \ref{Example4_4} (left) shows a comparison between the squared energy errors obtained by the uniform and adaptive refinements: the slopes for the adaptive approaches, in line with $O(DoF^{-1/2})$, are about twice those 
for the uniform one. This benefit is also reflected in the memory usage, see Figure \ref{Example4_4} (right). Despite the overhead for non-product meshes, memory grows more slowly with accuracy in the adaptive cases and memory savings are seen at higher accuracies for both indicators.

\section{Conclusions}

In this paper we have introduced a space-time adaptive boundary element method for acoustic soft-scattering problems, which are formulated as a weakly singular boundary integral equation.\\ 
The adaptive mesh refinements are steered by error indicators  based on the a posteriori error estimates of residual type in \cite{graded} for the $H^{0}_\sigma(\mathbb{R}^+,H^{-\frac{1}{2}}(\Gamma))$ error, respectively a heuristic modification. Compared to standard implementations for tensor product discretizations of the space-time cylinder $[0,T]\times\Gamma$, we have outlined algorithmic aspects including the efficient assembly of the Galerkin matrix for local tensor products, 
\begin{minipage}{13.cm}
\hspace{-0.6cm}
$\begin{array}{rl}\hspace{-0.2cm}
\begin{minipage}{6.4cm}
\includegraphics[trim=0 0 0 0, clip,scale=0.37]{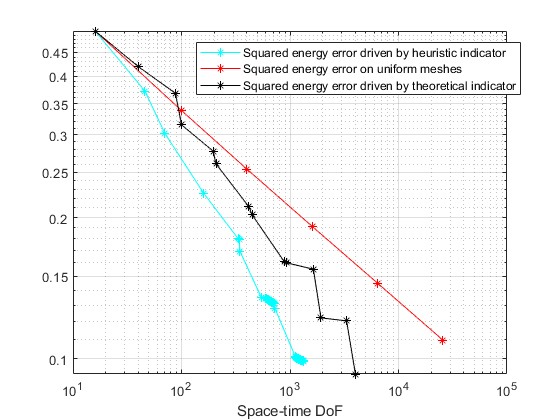}
\end{minipage}
&\hspace{0.2cm}
\begin{minipage}{6.4cm}
\includegraphics[trim=25 0 44 0, clip,scale=0.37]{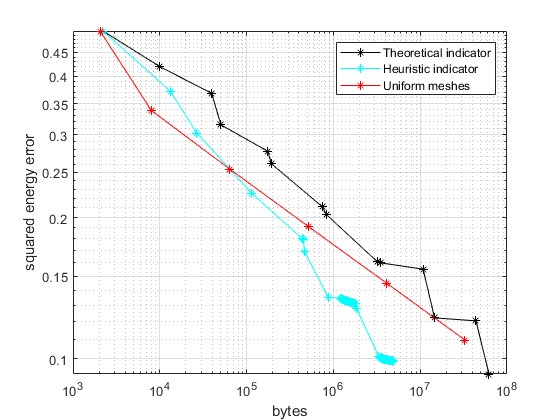}
\end{minipage}
\end{array}$
\captionof{figure}{Example \ref{example4}: comparison between the squared energy errors for uniform and adaptive refinements w.r.t.~$DoF\!s$ (left) and w.r.t.~memory consumption (right).}
\label{Example4_4}
\end{minipage}

\vspace*{0.4cm}

its update after mesh refinements, as well as the computation of error indicators.\\
The space-time adaptive algorithm has been studied in numerical experiments for wave scattering problems in $\mathbb{R}^2$ exhibiting a wide range of solutions with singularities in space, in time or in space-time. Numerical results show savings in $DoF\!s$ and in memory, and that the heuristic error indicator converges at the same rate as the energy error, suggesting its efficiency and reliability. Further, they confirm that the theoretical error indicator estimates a weaker norm than the energy norm, as expected from \cite{graded}, consistent with the expected norm of $H^{0}_\sigma(\mathbb{R}^+,H^{-\frac{1}{2}}(\Gamma))$, and for uniform refinements the theoretical error indicator leads to the expected convergence rates in all the experiments.\\
In the case of solutions with power-law singularities the obtained convergence rates on adaptively generated meshes are approximately twice of those obtained on uniform meshes. Like for time-independent problems, higher convergence rates can be achieved on (non-shape regular) graded meshes, but the proposed adaptive algorithm is limited by its shape-regular refinements. Anisotropic space-time refinements may therefore be relevant, but their theoretical basis remains widely open.\\
This work suggests the promise of an efficient space-time adaptive procedure for wave equations in $\mathbb{R}^3$, where larger savings in $DoF\!s$ and in memory are expected. 

{Beyond the model problem for the weakly singular integral equation addressed in the current work, future directions of interest include its extension to $\mathbb{R}^3$, to anisotropic mesh refinements and to $hp$ methods based on higher-order elements, as well as to nonlinear problems with complex, nonsmooth solutions \cite{aimi2023time,gimperlein2019space}. }\\


\textbf{Acknowledgments}: The authors wish to thank the Centre International de Rencontres Mathématiques (CIRM) in Marseille, France, for support of the research program \emph{Space-time adaptive boundary element methods for wave equations}. 

\bibliographystyle{siamplain}
\bibliography{Biblio}
\end{document}